\documentclass[reqno,12pt]{amsart}
\usepackage[all]{xy}

\let\a\alpha  \let\b\beta    
  
    \let\k\chi

\def\F{{\mathbb F}}
  
\def\Z{{\mathbb Z}}

\def\Q{{\mathbb Q}}

\def\pd{\prod}

\newtheorem{lemma}{Lemma}
\newtheorem{prop}{Proposition}
\newtheorem{theorem}{Theorem}

\theoremstyle{definition}
\newtheorem{defn}{Definition}

\theoremstyle{remark}
\newtheorem{rem}{Remark}
\newtheorem{example}{Example}
\usepackage[mathscr]{eucal}

\begin{document}

\title[Ergodic Theory]{Ergodic Theory Over ${\F}_2[[T]]$}
\author{Dongdai Lin}
\address{ The State Key Laboratory of Information Security\\
Institute of Software\\
Chinese Academy of Sciences \\
Beijing 100190, P.R.China.
}
\email{ddlin@is.iscas.ac.cn}
\author{Tao Shi}
\address{The State Key Laboratory of Information Security\\
Institute of Software\\
Chinese Academy of Sciences \\
Beijing 100190, P.R.China. AND
Graduate University of Chinese Academy of Sciences \\
Beijing 100049, P.R.China.}
\email{shitao@is.iscas.ac.cn}
\author{Zifeng Yang}
\address{Department of Mathematics\\
The Capital Normal University\\
Beijing 100048, P.R.China.
}
\email{yangzf@mail.cnu.edu.cn}

\date{Mar 12,2011}

\thanks{The research of the first author is partially supported by
``973 Program 2011CB302400", ``NSFC 60970152", and ``Institute of Software grand project YOCX285056".
The research of the third author is partially supported by
``NSFC 10871107" and a research program in mathematical automata ``KLMM0914".}

\begin{abstract}
In cryptography and coding theory, it is important to study the pseudo-random sequences
and the ergodic transformations. We already have the ergodic $1$-Lipshitz theory over ${\Z}_2$ established
by V. Anashin and others. In this paper we present an ergodic theory over
${\F}_2[[T]]$ and some ideas which might be very useful in applications.
\end{abstract}

\maketitle

Keywords: Ergodic; Function Fields.

\setcounter{section}{+0}

%======================================================================================================
\section{introduction}\label{intro}
A dynamical system on a measurable space $\mathbb{S}$ is understood
as a triple $(\mathbb{S};\mu;f)$, where $\mathbb{S}$ is a set
endowed with a measure $\mu
$, and%
\[
f:\mathbb{S\rightarrow S}%
\]
is a measurable function, that is, the $f$-preimage of any measurable
subset is a measurable subset.

A trajectory of the dynamical system is a sequence%
\begin{equation}\label{e:cycle}
x_{0}, f(x_{0}), f^{(2)}(x_{0}), f^{(3)}(x_{0}), \cdots
\end{equation}
of points of the space $\mathbb{S}$, $x_{0}$ is called an initial
point of the trajectory. If ${\mathbb S}$ is a finite set of $k$
elements and every element of ${\mathbb S}$ appears in the trajectory (\ref{e:cycle}),
then the finite sequence
\[
x_{0}, f(x_{0}), f^{(2)}(x_{0}), f^{(3)}(x_{0}), \cdots, f^{(k-1)}(x_0)
\]
is a single cycle (also called single orbit) of ${\mathbb S}$ given by the function $f$.

\begin{defn}
A mapping $F:\mathbb{S}\rightarrow\mathbb{Y}$ of a measurable space
$\mathbb{S}$ into a measurable space $\mathbb{Y}$ endowed with
probabilistic measure $\mu$ and $\nu$, respectively, is said to be
measure-preserving whenever $\mu(F^{-1}(S))=\nu(S)$ for each
measurable subset $S\subset$ $\mathbb{Y}$. In case ${\mathbb S}=$
$\mathbb{Y}$ and $\mu=\nu$, a measure preserving mapping $F$ is said
to be ergodic if $F^{-1}(S)=S$ for a measurable set $S$ implies
either $\mu(S)=1$ or $\mu(S)=0$.
\end{defn}

In non-Archimedean dynamics, the measurable space ${\mathbb S}$
would be taken to be a complete discrete valuation ring. There are
two kinds of such rings as follows,
\begin{itemize}
\item[(1)] characteristic $0$ case: finite extensions ${\mathcal O}$ of rings ${\Z}_p$ of $p$-adic integers for
prime numbers $p$\,;
\item[(2)] positive characteristic case: rings ${\F}_r[[T]]$ of formal power series over finite fields
${\F}_r$, where ${\F}_r$ is the field of $r$ elements and $r$ is a
power of some prime number $p$.
\end{itemize}
Both of these two kinds of complete rings are commutative and
compact, so these rings are certainly equipped with Haar measures with respect to
addition, which we always
assume to be normalized so that the measure of the total space in all cases is
equal to $1$.

In the characteristic $0$ case, much has been studied when the space
is the ring ${\Z}_p$ of $p$-adic integers, especially when the prime
number $p=2$. On ${\Z}_p$, the normalized measure $\mu$ and $p$-adic
absolute value $|\cdot|_p$ are related by the following: for any
$0\ne a \in {\Z}_p$, we can express it as $a=p^{{\rm ord}(a)}\cdot
u$ with ${\rm ord}(a)$ a non-negative integer and
$u=\sum_{i=0}^\infty u_i\, p^i$, where $u_i\in \{0, 1, 2, \cdots, p-1\}$
for each $i$ and $u_0\ne 0$, then $\mu(a{\Z}_p)=(1/p)^{{\rm
ord}(a)}=|a|_p$. To study measure preserving functions on ${\Z}_p$,
we at first consider $1$-Lipschitz functions. A function
$f:{\Z}_p\to {\Z}_p$ is said to satisfy the $1$-Lipschitz condition
if
\[
|f(x+y)-f(x)|_p\le |y|_p, \text{ for any } x,y\in {\Z}_p.
\]
Such a condition is also called ``compatibility" condition. It is clear
that a $1$-Lipschitz function on ${\Z}_p$ is continuous.

A continuous function $f: {\Z}_p\to {\Z}_p$  has its Mahler
expansion:
\[
f(x)=\sum_{i=0}^{\infty}a_{i}\dbinom{x}{i}, \quad\text{ with
$a_{i}\in\mathbb{Z}_{p}$\,,  and $a_i\to 0$  as  $i \to \infty$},
\]
where $\binom{x}{i}=x(x-1)\cdots(x-i+1)/i$ for an integer $i\ge 1$
and $\binom{x}{0}=1$. In the work of
\cite{An1}\cite{An1998}\cite{An2}\cite{An3}\cite{An4}, some
sufficient and necessary conditions on the Mahler coefficients for
$f$ to be $1$-Lipschitz or measure-preserving are given. When $p$ is
odd (respectively, $p=2$), the sufficient (respectively, sufficient
and necessary) conditions on the Mahler coefficients and the Van der
Put coefficients are also given  for $f$ to be ergodic in these
papers.

Ergodic functions enjoy many applications in cryptography theory. At
the beginning of 1990-th, V. Anashin \cite{An1} (1994) studied
differentiable functions on ${\Z}_p$ and gave criteria for
measure-preservation and ergodicity of $1$-Lipschitz transformations
on ${\Z}_2$. He did further work in the theory of
measure-preservation and ergodicity of $1$-Lipschitz functions on
${\Z}_p$ in \cite{An2} (2002). The ergodic theory for $1$-Lipschitz
transformations on the ring of $2$-adic integers is of
cryptographical significance in application. A. Klimov and A. Shamir
\cite{KS2002}\cite{KS2003}\cite{KS2004} proposed the concept of
$T$-function and used $T$-functions to produce long period
pseudo-random sequences, some of which are used as primitive
components in stream ciphers. Mathematically, these
results are derivable from V. Anashin's work \cite{An1}. Most
recently, V. Anashin et al \cite{An4} have used the Van der Put
basis to describe the ergodic functions on ${\Z}_2$.

In the positive characteristic case, we need to consider the ring ${\F}_r[[T]]$ of formal power series
over ${\F}_r$. The absolute value $|\cdot|_T$ on ${\F}_r[[T]]$ is normalized by $|T|_T=1/r$.
Carlitz polynomial functions $G_i(x)$ ($i\ge 0$) (Carlitz polynomials are reviewed in section \ref{carlitz})
are analogues of the binomial functions $\binom{x}{i}$, and every continuous function
$f: {\F}_r[[T]]\to {\F}_r[[T]]$ can be expressed as
\[
f(x)\sum\limits_{i=0}^\infty a_i G_i(x), \quad\text{ where $a_i\in {\F}_r[[T]]$, and $a_i\to 0$ as $i\to \infty$}.
\]
There are also analogous results on the theory of differentiability
of functions over ${\F}_r[[T]]$, for example, in \cite{Wa}\cite{Ya}.
Hence some natural questions would be asked: do we also have a
measure-preservation and ergodicity theory over ${\F}_r[[T]]$ or
${\F}_2[[T]]$? and if such a theory exists, how does it compare to
the theory over ${\Z}_p$? The purpose of this paper tries to answer
the first question, that is, to establish a theory of
measure-preservation and ergodicity over ${\F}_2[[T]]$.

In this paper, section \ref{prelim} will recall some of the existing
results on the theory of measure-preservation and ergodicity over
${\Z}_p$. In section \ref{ergodic}, we will give a description of
measure-preservation and ergodicity of $1$-Lipschitz functions over
${\F}_2[[T]]$ in terms of the Van der Put basis. In section
\ref{carlitz}, the conditions of a continuous function on
${\F}_r[[T]]$  to satisfy the $1$-Lipschitz condition are given in
terms of the Carlitz basis. Section \ref{ecarlitz} gives necessary
and sufficient conditions for a $1$-Lipschitz function on
${\F}_2[[T]]$ to be ergodic in terms of Carlitz basis. And finally, section \ref{discussion}
gives some discussion on the topics related to this paper.

The work of this paper is inspired by the ergodic theory established
by V. Anashin et al. The first author and second author want to
thank V. Anashin for many helpful talks about the topics of $p$-adic
dynamical systems and related problems. The authors also appreciate
the referees for many suggestions on the paper to make the
exposition better.

%=======================================================================================================================
\section{Preliminary}\label{prelim}

In this section, we will list some of the existing results on the theory of
measure-preservation and ergodicity of $1$-Lipschitz
transformations of ${\Z}_p$.

For a real number $x$, let $\lfloor x \rfloor$ denote the maximum
integer which is smaller than or equal to $x$.

\begin{prop}[\cite{An1}]\label{prop:1}
A $1$-Lipschitz function $f:\mathbb{Z}_{p}\rightarrow\mathbb{Z}_{p}$
is measure-preserving (respectively, ergodic) if and only if it is bijective
(respectively, transitive) modulo $p^{k}$ for all
integers $k\ge 0$.
\end{prop}

\begin{theorem}[\cite{An1}, \cite{An2}]\label{thm:An1}
(measure-preservation property) The function
$f:\mathbb{Z}_{p}\rightarrow\mathbb{Z}_{p}\,,$
\[
f(x)=\sum_{i=0}^{\infty}a_{i}\dbinom{x}{i}
\]
defines a measure-preserving $1$-Lipschitz transformation on
$\mathbb{Z}_{p}$ whenever the following conditions hold
simultaneously:
$$
\begin{array}{ll}
&a_{1}\not \equiv 0 \quad (\operatorname{mod} p),\\
&a_{i} \equiv0 \quad (\operatorname{mod} p^{\lfloor\log_{p}(i)\rfloor +1}), \quad
i=2,3,\cdots.
\end{array}
$$

The function f defines an ergodic $1$-Lipschitz transformation on
$\mathbb{Z}_{p}$ whenever the following conditions hold
simultaneously:
$$
\begin{array}{ll}
&a_{0}\not \equiv0 \quad (\operatorname{mod}p),\\
&a_{1}\equiv1 \quad (\operatorname{mod}p), \quad for \ p \ odd, \\
&a_{1}\equiv1 \quad (\operatorname{mod}4), \quad for \ p=2, \\
&a_{i}\equiv0 \quad (\operatorname{mod}p^{\lfloor \log_{p}(i+1)\rfloor +1}), \quad
i=2,3,\cdots.
\end{array}
$$

Moreover, in the case $p = 2$ these conditions are necessary.
\end{theorem}

For any non-negative integer $m$, the Van der Put function
$\chi(m,x)$ on ${\mathbb Z}_p$ is the characteristic function of the ball $B_{p^{-\lfloor \log_p(m)\rfloor -1}}(m)$
which is centered at $m$ and is of radius $p^{-\lfloor \log_p(m)\rfloor -1}$:
\[
\chi(m,x)=
\begin{cases}
1, & \quad \text{ if } |x-m|_p\le p^{-\lfloor \log_p (m)\rfloor -1},\\
0, & \quad \text{ otherwise, }
\end{cases}
\]
for $m\ne 0$, and is the characteristic function of the ball
$B_{1/p}(0)$ which is centered at $0$ and is of radius $1/p$:
\[
\chi(0,x)=
\begin{cases}
1, & \quad \text{ if } |x|_p\le 1/p,\\
0, & \quad \text{ otherwise .}
\end{cases}
\]
The Van der Put functions $\chi(m,x)$ constitute an orthonormal
basis of the space $C({\Z}_p, {\Q}_p)$ of the continuous functions
from ${\Z}_p$ to ${\Q}_p$ (see Theorem 1 to Theorem 4, Chapter 16, \cite{Ma}). In terms of the Van der
Put basis $\{\chi(m,x)\}_{m \ge 0}$, we have
\begin{theorem}[\cite{An4}]\label{thm:An4}
A function $f:\mathbb{Z}_{2}\rightarrow\mathbb{Z}_{2}$ is compatible
and preserves the Haar measure if and only if it can be represented
as
$$
f(x)=b_{0}\chi(0, x)+ b_{1}\chi(1, x)+
\sum\limits_{m=2}^{\infty}2^{\lfloor \log_{2}(m)\rfloor }b_{m}\chi(m, x),
$$
where $b_{m}\in \mathbb{Z}_{2}$ for $m=0,1,2,\cdots$, and
\begin{itemize}
\item $b_{0}+b_{1}\equiv1(\operatorname{mod}2)$

\item $\left\vert b_{m}\right\vert _{2}=1,$ for $m\geq2$.
\end{itemize}
\end{theorem}

\begin{theorem}[\cite{An4}]\label{thm:An4.2}
A function $f:\mathbb{Z}_{2}\rightarrow\mathbb{Z}_{2}$ is compatible
and ergodic if and only if it can be represented as
$$
f(x)=b_{0}\chi(0, x)+ b_{1}\chi(1, x)+
\sum\limits_{m=2}^{\infty}2^{\lfloor \log_{2}(m)\rfloor}b_{m}\chi(m, x),
$$
where $b_{m}\in \mathbb{Z}_{2}$ for $m=0,1,2,\cdots$, and
\begin{itemize}
\item $b_{0}\equiv1 \ (\operatorname{mod}2),\ b_{0}+b_{1}\equiv 3 \ (\operatorname{mod}4),
\ b_{2}+b_{3}\equiv2 \ (\operatorname{mod}4)$,

\item $\left\vert b_{m}\right\vert _{2}=1,$ for $m\geq2$,

\item
$\sum_{m=2^{n-1}}^{2^{n}-1}b_{m}\equiv0(\operatorname{mod}4)$, for
$n\geq3$.
\end{itemize}
\end{theorem}

The two complete discrete valuation rings $\mathbb{F}_{p}[[T]]$ and
$\mathbb{Z}_{p}$ are homeomorphic. Therefore in the same way we
define the Van der Put basis $\{\chi(\a, x)\}_{\a\in {\F}_p[T]}$ on
the space of functions over $\mathbb{F}_{p}[[T]]$:
\[
\chi(\alpha,x)=
\begin{cases}
%TCIMACRO{\QDATOPD{\{}{.}{1\text{ \ if }x\in B_{r^{-\deg(\alpha)}}(\alpha
%)}{0\text{ \ \ \ \ \ \ otherwise }}}%
%BeginExpansion
1, & \quad \text{ if }x\in B_{p^{-\deg_T(\alpha)-1}}(\a),\\
0, & \quad \text{ otherwise, }
\end{cases}
\]
is the characteristic function of the ball
\[
B_{p^{-\deg_T(\alpha)-1}
}(\alpha) =\{x\in {\F}_p[[T]]: \,\,\left\vert x-\alpha\right\vert
_{T}\leq p^{-\deg_T(\alpha )-1}\}
\]
which is centered at $\a$ and is of radius
$|T|_T^{\, \deg_T(\a)+1}$ if ${\a}\ne 0$, and
\[
\chi(0,x)=
\begin{cases}
1, & \quad \text{ if }x\in B_{p^{-1}}(0),\\
0, & \quad \text{ otherwise, }
\end{cases}
\]
is the characteristic function of the ball $B_{p^{-1}}(0)$ centered at $0$ and of radius $p^{-1}$.
In a similar way to the one used in the proof
of the Theorem 1 to Theorem 4 of Chapter 16 in Mahler's book \cite{Ma}, we can see that
every continuous function $f:\mathbb{F}_{p}[[T]]\rightarrow
\mathbb{F}_{p}[[T]]$ can be expressed as
\[
f(x)=\sum_{\alpha\in\mathbb{F}_{p}[T]}B_{\alpha}\chi(\alpha,x),\quad
B_{\alpha}\in\mathbb{F}_{p}[[T]], \,\,\, \text{ with } B_{\a}\to 0
\text{ as } \deg_T{\a}\to \infty.
\]

%==========================================================================================================
\section{Ergodic Functions over ${\F}_2[[T]]$ and Van der Put Expansions}\label{ergodic}
We will discuss the ergodic functions over ${\F}_2[[T]]$ and which conditions of
the expansion coefficients under the Van der Put
basis should be satisfied.

The absolute value on the discrete valuation ring ${\F}_2[[T]]$ is
normalized so that $|T|_T=1/2$.

Suppose $f: {\F}_2[T]\to {\F}_2[T]$ is a map. Then $f$ can be
expressed in terms of Van der Put basis:
\begin{equation}\label{e:2.1}
f(x)=\sum_{\a\in {\F}_2[T]} B_{\a}\chi(\a, x).
\end{equation}
If the function $f$ is continuous under the $T$-adic topology, then
$B_{\a}\to 0$ as $\deg_T(\a)\to \infty$, and $f$ can be extended to a
continuous function from ${\F}_2[[T]]$ to itself, which is still
denoted by $f$. The coefficients $B_{\a}$ of the expansion
(\ref{e:2.1}) can be calculated as follows (see Chapter 16 of Mahler's book \cite{Ma} for the detail):
\begin{itemize}
\item $B_0=f(0), B_1=f(1)$;
\item $B_{\a}=f(\a)-f({\a}-\a_n T^n)$, if ${\a}=\a_0+\a_1T+\cdots +\a_n T^n \in {\F}_2[T]$ with $\a_n\ne
0$ ( that is, $\a_n=1$ ) is of degree greater than or equal to $1$.
\end{itemize}
Also for ${\a}=\a_0+\a_1T+\cdots +\a_n T^n \in {\F}_2[T]$, we denote
by
\[
{\a}_{[k]}=\a_0 +\a_1 T + \cdots + \a_k T^k
\]
for any $k$ between $0$ and $n$.

\begin{theorem}
A continuous function $f:
\mathbb{F}_{2}[[T]]\rightarrow\mathbb{F}_{2}[[T]]$ is $1$-Lipschitz
if and only if it can be expressed as
\[
f(x)=b_0\chi(0,x)+\sum_{\alpha\in\mathbb{F}_{2}[T]\backslash\{0\}}T^{\deg_{T}(\alpha)
}b_{\alpha}\chi(\alpha,x), \quad \text{ with } b_{\a}\in
{\F}_2[[T]].
\]
\end{theorem}

\begin{proof}
Suppose $f$ is $1$-Lipschitz. Then it is clear that $b_0=B_0=f(0)\in
{\F}_2[[T]]$ and $b_1=B_1=f(1)\in {\F}_2[[T]]$. For
${\a}=\a_0+\a_1T+\cdots +\a_n T^n \in {\F}_2[T]$ of degree $n\ge 1$,
we have
\[
\left\vert B_{\alpha}\right\vert _{T}=\left\vert
f(\alpha)-f(\alpha-\alpha_{n} T^{n})\right\vert _{T}\leq\left\vert
\alpha_{n}T^{n}\right\vert _{T} =2^{-n}.
\]
Therefore $B_{\alpha}$ can be written as $B_{\alpha}=T^{\deg_{T}
(\alpha)}b_{\alpha}$ with $b_{\a}\in {\F}_2[[T]]$.

Conversely, suppose $f(x)=b_0\chi(0,x)+\sum_{\alpha\in\mathbb{F}_{2}
[T]\backslash \{0\}}T^{\deg_{T}(\alpha)}b_{\alpha}\chi(\alpha,x)$ with
$b_{\a}\in {\F}_2[[T]]$. If $X,Y$ both belong to
$\mathbb{F}_{2}[[T]]$ and $X\equiv Y(\operatorname{mod}T^{n})$, then
\[
\chi(\a,X)=\chi(\a,Y) \text{ for any $\a\in {\F}_2[[T]]$ with $\a=0$
or $\deg_T(\a)< n$.}
\]
Therefore we have $f(X)\equiv f(Y)(\operatorname{mod}T^{n})$, hence
the function $f$ is $1$-Lipschitz.
\end{proof}

\begin{prop}
A $1$-Lipschitz function
$f:\mathbb{F}_{p}[[T]]\rightarrow\mathbb{F}_{p}[[T]]$ is
measure-preserving (respectively, ergodic) if and only if $f$ is
bijective modulo $T^{k}$(respectively, transitive modulo $T^{k}$)
for all integers $k> 0$.
\end{prop}
\begin{proof}
There is a topological bijection
\[
\begin{array}{rrcl}
\sigma: &{\Z}_p&\to &{\F}_p[[T]]\\
        &    \sum_{i=0}^\infty c_i p^i &\mapsto &\sum_{i=0}^\infty c_i T^i,
\end{array}
\]
where $c_i\in\{0, 1, \cdots, p-1\}$ for each $i$. The map $\sigma$ preserves
the Haar measures, that is, if we denote the normalize Haar measure on ${\Z}_p$
by $\mu$ and that on ${\F}_p[[T]]$ by $\nu$, then $\mu(S)=\nu(\sigma(S))$ for
any measurable subset $S\subset {\Z}_p$. Therefore that $f: \mathbb{F}_{p}[[T]]\rightarrow\mathbb{F}_{p}[[T]]$
is measure-preserving is equivalent to saying that the function
$\sigma^{-1}\circ f\circ \sigma: {\Z}_p\to {\Z}_p$ is measure-preserving, which
is equivalent to saying that $\sigma^{-1}\circ f\circ \sigma$ is bijective modulo $p^k$ for
all positive integers $k$ by Proposition \ref{prop:1}, hence is also
equivalent to saying that $f:\mathbb{F}_{p}[[T]]\rightarrow\mathbb{F}_{p}[[T]]$
is bijective modulo $T^k$ for all positive integers $k$ since $\sigma$ is measure-preserving.
The proof about the ergodicity property is similar.
\end{proof}
\begin{theorem}\label{theorem:mp}\bigskip(measure-preservation property)
A $1$-Lipschitz function
$f:\mathbb{F}_{2}[[T]]\rightarrow\mathbb{F}_{2}[[T]]$,
\[
f(x)=b_0\chi(0,x)+\sum_{\alpha\in\mathbb{F}_{2}[T]\backslash\{0\}}T^{\deg_{T}(\alpha)
}b_{\alpha}\chi(\alpha,x), \quad \text{ with } b_{\a}\in {\F}_2[[T]]
\]
is measure preserving if and only if the following conditions hold
simultaneously:
\begin{itemize}
\item[(1)] $b_{0}+b_{1}\equiv1(\operatorname{mod}T)$;
\item[(2)] $\left\vert b_{\alpha}\right\vert_T =1,$ for $\deg_T(\alpha)\geq 1$.
\end{itemize}
\end{theorem}
\begin{proof}
Suppose $f$ is bijective $\operatorname{mod}T^{n}$ for all
$n\in\mathbb{N}$, we need to show that the two conditions for the
Van der Put coefficients are satisfied. At first, from the
bijectivity of $\operatorname{mod}T$, we get
\begin{itemize}
\item[] $f(0)=b_{0}\equiv 1\,(\operatorname{mod}T)$; $f(1)=b_{1}\equiv
0\,(\operatorname{mod}T)$; or
\item[] $f(0)=b_{0}\equiv 0\,(\operatorname{mod}T)$; $f(1)=b_{1}\equiv
1\,(\operatorname{mod}T)$.
\end{itemize}
Thus we get $b_{0}+b_{1}\equiv 1\,(\operatorname{mod}T)$. Secondly,
we consider the bijectivity of the function $f$ $\mod T^{2}$. As
$f(T)-f(0) \not\equiv 0(\operatorname{mod}T^{2})$, we get
$b_{0}\chi(0,T)+Tb_{T}\chi(T,T)-b_{0} \chi(0,0)=Tb_{T}\not\equiv
0(\operatorname{mod}T^{2})$, therefore $\left\vert b_{T}\right\vert_T
=1$; Also $f(1+T)-f(1)\not\equiv 0(\operatorname{mod}T^{2})$ implies
$b_{1}\chi(1,1+T)+Tb_{1+T}\chi(1+T,1+T)-b_{1}\chi(1,1)=Tb_{1+T}\not\equiv
0(\operatorname{mod}T^{2})$, therefore $\left\vert
b_{1+T}\right\vert_T=1$. In the same way for the general case when
$\deg_T(\alpha)=n$, we use the bijectivity of the function
$\operatorname{mod}T^{n+1}$, thus $f(\alpha)-f(\alpha-\alpha_{n}T^{n}%
)=T^{\deg_{T}(\alpha)}b_{\alpha}=T^{n}b_{\alpha}\not\equiv
0(\operatorname{mod}T^{n+1})$, therefore $\left\vert
b_{\alpha}\right\vert_T=1.$

Conversely, suppose the two conditions for Van der Put coefficients
are satisfied. As $f(0)=b_{0}, f(1)=b_{1}$, we see that the first
condition implies the bijectivity $\operatorname{mod}T$. To derive
the bijectivity  $\operatorname{mod}T^{n}$, $n\ge 2$, we choose
\[
X=X_0 + X_1 T + \cdots + X_{n-1} T^{n-1}, \quad Y=Y_0 + Y_1 T +
\cdots + Y_{n-1}T^{n-1}
\]
such that $f(X)-f(Y)\equiv 0(\operatorname{mod}T^{n})$. If $X\not\equiv
Y(\operatorname{mod}T^{n})$, then we denote the first integer $m$
between $0$ and $n-1$ such that $X_{m}\neq Y_{m}$ and consider the
equation $f(X)-f(Y)\equiv 0(\operatorname{mod}T^{m+1})$. Then we can
assume $X_{[m]}=X_{[m-1]}+T^m$ and $Y_{[m]}=X_{[m-1]}$ for convenience.
Therefore the coefficient $B_{\a}$ of the expansion of $f(x)$ at ${\a}=X_{[m]}$ is
\[
B_{X_{[m]}}=f(X_{[m]})-f(X_{[m-1]})=f(X_{[m]})-f(Y_{[m]})\equiv 0 (\operatorname{mod}T^{m+1}),
\]
which contradicts to $B_{X_{[m]}}=T^m\cdot b_{X_{[m]}}\not\equiv 0 (\operatorname{mod}T^{m+1})$.
Therefore $X\equiv
Y(\operatorname{mod}T^{n})$, and so $f$ is injective. Hence $f$ is
bijective $\operatorname{mod}T^{n}$, as $\mathbb{F}_{2}[[T]]/(T^{n}\cdot {\F}_2[[T]])$
is a finite set.
\end{proof}

For any $x\in {\F}_2[[T]]$, we write
\begin{equation}\label{e:exp}
x=\sum_{i=0}^\infty x_i T^i.
\end{equation}
\begin{lemma}\label{lem:2.1}
Suppose a measure-preserving $1$-Lipschitz function $f:\mathbb{F}_{2}[[T]]\rightarrow\mathbb{F}%
_{2}[[T]]$ is transitive (single orbit) over
$\mathbb{F}_{2}[[T]]/(T^{n}\cdot {\F}_2[[T]]),$ $n\ge 1$. Then $f$
is transitive over $\mathbb{F}_{2}[[T]]/(T^{n+1}\cdot {\F}_2[[T]])$
if and only if  $\,\,\,\#\{$$x$$\in
$${\mathbb F}_{2}[T]$$\,\,:\,\,{\deg}_T(x)$$<$$n$ and $
f(x)_n=1\}$ is an odd integer, where the notation $f(x)_n$ is as (\ref{e:exp}) since
$f(x)\in {\F}_2[[T]]$.
\end{lemma}

\begin{rem}\label{r:1.1}
We are going to give descriptions of ergodic functions on ${\F}_2[[T]]$ in terms
of Van der Put basis (Theorem \ref{theorem:put}) and in terms of Carlitz basis
(Theorem \ref{theorem:car}). But in applications to computer programming of
cryptography, the proof of this lemma could also provide a method in creating
single cycles modulo $T^k$ for a given positive integer $k$, see section \ref{discussion}
for the discussion.
\end{rem}

\begin{proof} Let $A={\F}_2[T]$ be the polynomial ring over ${\F}_2[T]$, and $A_{<n}=\{x\in
A:\ \deg_T(x)<n\}$ for any non-negative integer $n$.

``Necessity". As $f$ is $1$-Lipschitz, when we consider the
trajectory of $f$ modulo $T^{k}$, we need only consider
$\{x_{0}\operatorname{mod}T^{k},f(x_{0}\operatorname{mod}
T^{k}),\cdots,f^{(i)}(x_{0}\operatorname{mod}T^{k}),\cdots\}$ with
representatives of image elements chosen in $A_{<k}$. If $f$ is
transitive over $\mathbb{F}_{2}[[T]]/(T^{n+1}\cdot {\F}_2[[T]])$,
then there exist $x_{0},x_{1}\in A_{<n}$ such that $f(x_0)=x_1+T^n$.
We consider the trajectory of $f$ modulo $T^{n+1}$ starting with
$x_0$:
\begin{equation}\label{e:2.2}
\begin{array}{cccccccc}
x_{0} & \rightarrow  & f(x_0) &  \rightarrow &\cdots & \rightarrow
&f^{(2^n-1)}(x_{0})& \rightarrow\\
 \rightarrow f^{(2^n)}(x_0) & \rightarrow &f^{(2^n+1)}(x_0)
&\rightarrow &\cdots &\rightarrow &f^{(2^{n+1}-1)}(x_{0})
&\rightarrow\\
\rightarrow x_{0}+T^{n+1} (1+\ast)
 &&&&(\mod T^{n+1})&&&
\end{array}
\end{equation}
where $``\ast" \in T\mathbb{F}_{2}[[T]]$, and an element of the
second row is equal to the corresponding element of the first row in
the the column plus a $T^n$, since the map $f$ is measure preserving and transitive
$(\operatorname{mod}T^{n+1})$: $f^{(2^n+i)}(x_0) \equiv f^{(i)}(x_0)+T^n \mod
T^{n+1}$ for $0\le i \le 2^n-1$. We look at the elements from the
left to the right in the first row, if there is an element in $A_{<n}$
other than $x_0$ mapped by $f$ to an element in the set $A_{<n}+T^n$,
then there would be some other element in the set $A_{<n}+T^n$ mapped
to $A_{<n}$, and hence in the second row there would be an element in
$A_{<n}$ mapped by $f$ to an element in $A_{<n}+T^n$. This implies that
the total number of elements in $A_{<n}$ in the trajectory mapped by
$f$ to $A_{<n}+T^n$ is an odd integer, that is, $\#\{$$x$$\in
$${\mathbb F}_{2}[T]$$\,\,:\,\,{\deg}_T(x)$$<$$n$ and $
f(x)_n=1\}$ is an odd integer.

``Sufficiency". By the condition, there must exist $x_0, x_1\in A_{<n}$
such that $f(x_0)=x_1+T^n+{\rm higher\,\,\, terms}$. We consider diagram (\ref{e:2.2}) again.
Since $f$ is transitive modulo $T^n$, the elements of the first row
are distinct and so are the elements of the second row. It also
implies that $f^{(2^n)}(x_0)$ is either equal to $x_0$ or $x_0+T^n$ $(\operatorname{mod}T^{n+1})$.
But if $f^{(2^n)}(x_0)=x_0$, then $\#\{$$x$$\in
$${\mathbb F}_{2}[T]$$\,\,:\,\,{\deg}_T(x)$$<$$n$ and $
f(x)_n=1\}$ would be an even integer. Therefore we must have
$f^{(2^n)}(x_0)=x_0+T^n$ $(\operatorname{mod}T^{n+1})$. And we get $f^{(2^n+i)}(x_0) \equiv
f^{(i)}(x_0)+T^n \mod T^{n+1}$ for $0\le i\le 2^n-1$ by the $1$-Lipschitz
measure-preserving assumption on $f$. Hence all the
elements in the first row and the second row of diagram (\ref{e:2.2})
are distinct, that is, $f$ is transitive modulo $T^{n+1}$.
\end{proof}

\begin{theorem}\label{theorem:put}
\bigskip(ergodicity property) A $1$-Lipschitz function $f:\mathbb{F}_{2}[[T]]\rightarrow
\mathbb{F}_{2}[[T]]$
\begin{equation}\label{e:put}
f(x)=b_0\chi(0,x)+\sum_{\alpha\in\mathbb{F}_{2}[T]\backslash\{0\}}T^{\deg_{T}(\alpha)
}b_{\alpha}\chi(\alpha,x), \quad \text{ with } b_{\a}\in {\F}_2[[T]]
\end{equation}
is ergodic if and only if the following conditions hold
simultaneously:

\begin{itemize}
\item[(1)] $b_{0}\equiv1\ (\operatorname{mod}T)$, $b_{0}+b_{1}\equiv
1+T\ (\operatorname{mod}T^{2}),$ $b_{T}+b_{1+T}\equiv T\ (\operatorname{mod}%
T^{2})$;

\item[(2)] $\left\vert b_{\alpha}\right\vert_T =1,$ for $\deg_T(\alpha)\geq1$;

\item[(3)] $\sum_{\deg_T(\alpha)=n-1}b_{\alpha}\equiv
T(\operatorname{mod}T^{2}), n\ge 2$.
\end{itemize}
\end{theorem}

\begin{proof}
Since $f$ is a $1$-Lipschitz function, we have
\[
f(x)=B_0\chi(0,x) +\sum_{\a\in{\F}_2\backslash\{0\}}
B_{\a}\chi(\a,x)=b_0\chi(0,x) +\sum_{\a\in{\F}_2[T]\backslash\{0\}}
T^{\deg_T({\a})} b_{\a}\chi(\a,x)
\]
with $b_{\a}\in {\F}_2[[T]]$.

 ``Necessity". Suppose $f$ is
ergodic. By transitivity modulo $T$, we get
\[
f(0)\equiv 1 \mod T, \quad f(1) \equiv 0 \mod T.
\]
Therefore $b_0=B_0\equiv 1 \ (\mod T)$, and also $f(0)+f(1)\equiv 1
\ (\mod T)$. But $f(0)+f(1)\neq 1(\operatorname{mod}T^{2})$,
otherwise we have $f(0)\equiv 1(\operatorname{mod}T^{2})$,
$f(1)\equiv0(\operatorname{mod}T^{2})$; or
$f(0)\equiv1+T(\operatorname{mod}T^{2})$, $f(1)\equiv
T(\operatorname{mod} T^{2})$, but by the transitivity
$\operatorname{mod}T^{2}$, these two cases can not appear. So we get
\begin{equation}\label{e:2.3}
f(0)+f(1)  \equiv1+T(\operatorname{mod}T^{2}), \quad \text{ that
is,}\quad b_{0}+b_{1}   \equiv1+T(\operatorname{mod}T^{2}).
\end{equation}
By the transitivity $\operatorname{mod}T^{2}$ and Lemma
\ref{lem:2.1}, we know that
\[
f(0)+f(1)+f(T)+f(1+T)  \equiv T^{2}(\operatorname{mod}T^{3}),
\]
which gives us
\[
B_{T}+B_{1+T}\equiv T^{2} (\operatorname{mod}T^{3}).
\]
As $B_{T}= Tb_{T}$, $B_{1+T}= Tb_{1+T}$, we get
\begin{equation*}
b_{T}+b_{1+T}\equiv T (\operatorname{mod}T^{2}).
\end{equation*}

Now consider
\begin{equation*}
\sum_{x\in A_{<n}}f(x) = \sum_{\beta\in A_{<n-1}
}f(\beta+T^{n-1})+\sum_{\beta\in A_{<n-1}}f(\beta)
  = \sum_{\deg_T(\alpha)=n-1}B_{\alpha}.
\end{equation*}
Lemma \ref{lem:2.1} gives us
\[
\sum_{x\in A_{<n}}f(x)\equiv T^{n}\operatorname{mod}T^{n+1},
\]
so we put these equations together to get
\[
\sum_{\deg_T(\alpha)=n-1}b_{\alpha}  \equiv
T\operatorname{mod}T^{2}.
\]

``Sufficiency". Suppose the three conditions are satisfied, we want
to prove $f$ is transitive on every ${\F}_{2}[[T]]/(T^{n}\cdot{\F}_2[[T]])$ for all
$n\in\mathbb{N}$. But this is just to apply Lemma \ref{lem:2.1} on
the induction process for $n$, the first condition gives the first
step of the induction.
\end{proof}

%====================================================================================================================
\section{$1$-Lipschitz Functions over ${\F}_r[[T]]$ and Carlitz
Expansions}\label{carlitz}

We first recall some useful formulas in function field
arithmetic, with all the details and expositions in Chapter 3 of \cite{Go}. Let
$A=\mathbb{F}_{r}[T]$ ($r$ is a power of the prime number $p$) with the normalized absolute value
$|\cdot|_T$ such that $|T|_T= 1/r$. The completion of $A$ with respect to this absolute value is
$\hat{A}={\F}_r[[T]]$.
\begin{defn}
We set the following notations:
\begin{itemize}
\item
$[i]=T^{r^{i}}-T$, where $i$ is a positive integer;
\item $L_i=1$ if
$i=0$; and $L_i=[i]\cdot [i-1]\cdots[1]$ \ if $i$ is
      a positive integer;
\item $D_i=1$ if $i=0$; and $D_i=[i]\cdot [i-1]^r\cdots [1]^{r^{i-1}}$
      \ if $i$ is a positive integer;
\item for any non-negative integer
$n=n_0+n_1\, r+\cdots+n_s\, r^s$, $0\le n_j\le r-1$ for $0\le j\le s$, the $n$-th Carlitz factorial
$\Pi(n)$ is defined by
\[
\Pi(n)=\pd_{j=0}^s D_j^{n_j};
\]
\item
$ e_d(x)=
\begin{cases}
x, &\quad\text{ if } d=0,\\
\prod\limits_{\a\in A,\,\deg_{T}(\alpha )<d}(x-\alpha), &\quad\text{
if $d$ is a positive integer;}
\end{cases}
$
\item $E_{i}(x)=e_{i}(x)/{D_{i}}$, for any non-negative integer $i$;
\item $G_{n}(x)=\prod\limits_{i=0}^{s}(E_{i}(x))^{n_{i}}$, $n=n_{0}+n_{1}r+\cdots
+n_{s}r^{s}$ non-negative integers, $0\leq n_{i}<r$;
\item
$G_{n}^{\prime}(x)=\prod\limits_{i=0}^{s}G_{n_{i}r^{i}}^{\prime}(x)$,
where $G^{\prime}_{n_i r^i}=
\begin{cases}
\left(E_i(x)\right)^{n_i}, &\quad\text{ if } 0\le n_i<r-1,\\
\left(E_i(x)\right)^{n_i}-1, &\quad\text{ if } n_i=r-1.
\end{cases}
$
\end{itemize}
\end{defn}
The polynomials $G_n(x)$ and $G'_n(x)$ are called Carlitz
polynomials.

\begin{prop}[\cite{Ca}]\label{prop:3.1}
The following formulas hold for the Carlitz polynomials
\begin{itemize}
\item $G_{m}(t+x)=\sum\limits_{\substack{k+l=m\\k,\,l\geq0}}\binom{m}{k}%
G_{k}(t)G_{l}(x),$ $t,x\in\mathbb{F}_{2}[[T]]$.
\item $G_{m}^{\prime}(t+x)=\sum\limits_{\substack{k+l=m\\k,\,l\geq0}}\binom
{m}{k}G_{k}(t)G_{l}^{\prime}(x)$.
\end{itemize}
Orthogonality property of $\{G_{n}(x)\}_{n\geq0}$ and
$\{G_{n}^{\prime }(x)\}_{n\geq0}$:
\begin{itemize}
\item For any $s<r^{m}$, $l$ an arbitrary non-negative integer$,$
\begin{equation}\label{e:3.0}
\sum_{\deg_T(\alpha)<m}G_{l}(\alpha)G_{s}^{\prime}(\alpha)=
\begin{cases}
0, &\quad\text{ if } l+s\ne r^m-1;\\
(-1)^m, &\quad\text{ if } l+s = r^m-1.
\end{cases}
\end{equation}
\end{itemize}
\end{prop}

The polynomials $G_{n}(x)$ and $G_{n}^{\prime}(x)$ map $A$ to $A$.
And it is well known that $\{G_{n}(x)\}_{n\ge 0}$ is an orthonormal
basis of the space $C(\mathbb{F}_{r}[[T]], \mathbb{F}_{r}((T)))$ of
continuous functions from ${\F}_r[[T]]$ to ${\F}_r((T))$, that is,
every $T$-adic continuous function can be written as:
\begin{equation*}
f(x)=\sum_{n=0}^{\infty}a_{n}G_{n}(x),\quad\text{ where $a_n\in
{\F}_r((T))$, \,\, and }  a_{n} \rightarrow 0,\text{ as }
n\rightarrow\infty,
\end{equation*}
with the sup-norm $||f||=\max\limits_{n}\{|a_n|_T\}$. Moreover, the
expansion coefficient $a_n$ can be calculated as:
\begin{equation}\label{e:car}
a_{n} =(-1)^{m}\sum_{\deg_T(\alpha)<m}G_{r^{m}-1-n}^{\prime
}(\alpha)f(\alpha),\text{ for any integer such that }r^{m}>n.
\end{equation}

Following Wagner~\cite{Wa}, we define a new sequence of polynomials
$\{H_n(x)\}_{n\geq 0}$ by
\begin{equation*}
\begin{split}
&\ H_0(x)=1,\qquad \text{ and } \\
&\ H_n(x)=\frac{\Pi(n+1)G_{n+1}(x)}{\Pi(n)\, x} \qquad\text{ for
     $n\geq 1$. }
\end{split}
\end{equation*}
Then we get
\begin{lemma}[\cite{Wa}]\label{lem:3.1}
 $\{H_n(x)\}_{n\geq 0}$ is an orthonormal basis of $C({\F}_r[[T]], {\F}_r((T)))$.
\end{lemma}

To study the $1$-Lipschitz functions over $\hat{A}={\F}_r[[T]]$, we recall the interpolation polynomials
introduced by Amice \cite{Am}. These polynomials $\{Q_n(x)\}_{n\ge 0}$ are constructed from
which is called by Amice the ``very well distributed sequence"  $\{u_n\}_{n \ge 0}$ with $ u_n\in A$:
\begin{equation}\label{e:3.1}
Q_n(x)=
\begin{cases}
1, & \text{ if } n=0,\\
\dfrac{(x-u_0)(x-u_1)\cdots (x-u_{n-1})}{(u_n-u_0)(u_n-u_1)\cdots (u_n-u_{n-1})}, &\text{ if } n\ge 1.
\end{cases}
\end{equation}
We choose the sequence $\{u_n\}_{n\ge 0}$ in the following way
such that $u_n\ne 0$ for any $n\ge 0$. Let
$S=\{\a_0,\a_1,\cdots,\a_{r-1}\}$ be a system of representatives of
$A/(T\cdot A)$, and assume that $\a_0=T$ (thus $0\not\in S$). Then
any element $x\in {\F}_r((T))$ can be uniquely written as
\[
x=\sum_{k \ge k_0}^\infty \b_k\,T^k
\]
where $\b_k\in S$ and $k_0\in \Z$, with $x$ in $\hat{A}$ if and only if
$\b_k=0$ for all $k<0$. To each non-negative integer
$n=n_0+n_1q+\cdots+n_s r^s$ in $r$-digit expansion, we
assign the element
\[
u_n=\a_{n_0}+\a_{n_1}T+\cdots+\a_{n_s}T^s.
\]
We have

\begin{theorem}[\cite{Am}]\label{thm:Am}
The interpolation polynomials $\{Q_n(x)\}_{n\ge 0}$ defined above is
an orthonormal basis of $C({\F}_r[[T]],{\F}_r((T)))$. That is, any
continuous function $f(x)$ from ${\F}_r[[T]]$ to ${\F}_r((T))$ can
be written as
\begin{equation}\label{e:3.2}
f(x)=\sum\limits_{n=0}^\infty a_n Q_n(x),
\end{equation}
where $a_n\to 0$ as $n\to \infty$, and the sup-norm of $f$ is given
by $||f||=\max\limits_n\{|a_n|_T\}$.
\end{theorem}
Since $Q_n(u_n)=1$ for all $n$, and $Q_m(u_n)=0$ for $m>n$ by the
equation (\ref{e:3.1}), we see that the expansion coefficients $a_n$
can be deduced by the following induction formula
\begin{equation}\label{e:3.3}
\begin{array}{ll}
&a_0=f(0)\\
&a_n=f(u_n)-\sum\limits_{j=0}^{n-1} a_j Q_j(u_n) \quad\text{ for $n\ge 1$ }.
\end{array}
\end{equation}
Equation (\ref{e:3.2}) is valid not only for continuous functions on
$\hat{A}={\F}_r[[T]]$, but also for any function $f$ from
$A\backslash \{0\}$ to ${\F}_r((T))$. Elements of the sequence
$\{u_n\}$ constitute the set $A\backslash\{0\}$, so the summation of equation (\ref{e:3.2}) is a
finite sum for $x\in A\backslash\{0\}$. The
element $0$ is excluded because of the way the sequence $\{u_n\}$ is
chosen. If the function $f$ is
continuous on $\hat{A}\backslash\{0\}$, then $f$ is certainly
determined by the values of $f$ at the points of $A\backslash\{0\}$.

Suppose $\{R_n(x)\}_{n \ge 0}$ is any orthonormal basis of $C(\hat{A},{\F}_r((T)))$
consisting of polynomials in the variable $x$ with $\deg_T(R_n)=n$. Then we have
\begin{equation}\label{e:3.4}
Q_n(x)=\sum\limits_{j=0}^n \gamma_{n,j} R_j(x),
\end{equation}
where $\gamma_{n,j}\in\hat{A}$ for all $n,j$, and
$\gamma_{n,n}\in {\hat{A}}^{\times}$.

For any positive integer $n$, write $n=n_0+n_1r+\cdots+n_wr^w$ in
$r$-digit expansion, with $n_w\neq 0$,
\begin{itemize}
  \item denote $\nu(n)$ the largest integer such
        that $r^{\nu(n)}|n$;
  \item $l(n)=l_r(n)=n_wr^w$.
\end{itemize}
\begin{lemma}\label{lem:3.2}
Let $n$ be a positive integer, then
\[
\frac{\Pi(n-1)}{\Pi(n)}=\frac{1}{L_{\nu(n)}}
\]
\end{lemma}
\begin{proof}
Straightforward computation.
\end{proof}

\begin{lemma}[\cite{Ya}]\label{lem:3.3}
If a function $f:{\F}_r[[T]]\backslash\{0\} \to {\F}_r((T))$ can be
expressed as $f(x)=\sum\limits_{n=0}^\infty a_n H_n(x)$ for all
$x\in {\F}_r[[T]], x\neq 0$, that is, the summation converges to
$f(x)$ for $x\neq 0$, then the sequence $\{a_n\}_{n\geq 0}$ is
determined by the values $f(x)$ for all $x\in
{\F}_r[T]\backslash\{0\}$. More precisely, for any non-negative
integer $n$, we choose an integer $w$ such that $n<r^w-1$ and set
$S=\{\alpha\in {\F}_r[T]:\,\, \deg_T(\a)<w, \a\neq 0\}$, then $a_n$ is
determined by the values of $f$ at the points of $S$:
\begin{equation}\label{e:coefficient}
a_n=\dfrac{(-1)^w}{L_{\nu(n+1)}}\sum\limits_{\a\in S}\a f(\a)G'_{r^w-2-n}(\a).
\end{equation}
\end{lemma}
\begin{proof}
This is a refined statement of Lemma 5.5 of \cite{Ya}. By definition and Lemma \ref{lem:3.2},
\[
H_n(x)=\dfrac{\Pi(n+1) G_{n+1}(x)}{\Pi(n) x}= L_{\nu(n+1)}\dfrac{G_{n+1}(x)}{x}
\]
for $n\ge 0$, thus
\[
xf(x)=\sum\limits_{n=0}^\infty a_n L_{\nu(n+1)}G_{n+1}(x)
\]
for any $x\neq 0$ in $\hat{A}$. Since $G_{n+1}(0)=0$, we sum up all elements $\a\in S$ in
the above equation, and apply the equation (\ref{e:3.0}) of orthogonality property to get
for any non-negative integer $m<r^w-1$
\[
\begin{array}{ll}
&\sum\limits_{\a\in S}\a f(\a)G'_{m}(\a)=\sum\limits_{\a\in S}\sum\limits_{n\ge 0}
a_n L_{\nu(n+1)} G_{n+1}(\a)G'_m(\a)\\
=&\sum\limits_{\deg_T(\a)<w}\sum\limits_{n\ge 0} a_n L_{\nu(n+1)}G_{n+1}(\a)G'_m(\a)\\
=&\sum\limits_{n\ge 0}a_n L_{\nu(n+1)}\sum\limits_{\deg_T(\a)<w}G_{n+1}(\a)G'_m(\a)\\
=&(-1)^w a_{r^w-2-m} L_{\nu(r^w-1-m)}.
\end{array}
\]
Notice that the summations on $n$ are actually finite sums, thus we can change the order of
summations on $\a$ and on $n$. Therefore we get the formula (\ref{e:coefficient}), and the
conclusion.
\end{proof}

\begin{lemma}\label{lem:3.4}
We have
\begin{itemize}
\item[(1)]
  $
    |L_n|_T=r^{-n}=|T|_T^{n}
  $
for any non-negative integer $n$;
\item[(2)] For any non-negative integer $n$, $\nu(n)\le \lfloor \log_r (n)\rfloor$.
\end{itemize}
\end{lemma}
\begin{proof}
Immediate from definition.
\end{proof}

Denote
\[
((i_1,i_2,\cdots,i_s))=\frac{(i_1+i_2+\cdots+i_s)!}{i_1!\:
i_2!\:\cdots\:i_s!}
\]
for any integers $i_1,i_2,\cdots,i_s\geq 0$. We
have the following assertion about the multinomial
numbers by Lucas \cite{Lu}:
\begin{lemma}[Lucas]\label{lucaslemma}
For non-negative integers $n_0,n_1,\cdots,n_s$,
\begin{equation}\label{e:lucaslemma}
((n_0,n_1,\cdots,n_s)) \equiv \pd_{j\geq 0}
  ((n_{0,j},n_{1,j},\cdots,n_{s,j})) \mod p
\end{equation}
where $n_i=\sum\limits_{j\geq 0}n_{i,j}\, r^j$ is the
$r$-digit expansion for $i=0,1,\cdots,s$.
\end{lemma}
\begin{rem}\label{r:3.1}
Lemma~\ref{lucaslemma} is useful when $s=1$.
In this case formula (\ref{e:lucaslemma}) is
expressed in the form: let $n=\sum\limits_j n_j\, r^j$ and
$k=\sum\limits_j k_j\, r^j$ be $r$-digit expansion for
non-negative integers $n$ and $k$, then
\[
\binom nk \equiv \pd_{j\geq 0} \binom{n_j}{k_j} \mod p.
\]
\end{rem}

\begin{lemma}\label{lem:3.5}
Let $f(x)=\sum\limits_{n=0}^\infty a_n H_n(x)$ be a continuous
function from $\hat{A}\backslash\{0\}$ to ${\F}_r((T))$ ( this
implies that the series converges for any $x\in
\hat{A}\backslash\{0\}$ ). Suppose that $|f(x)|_T\le 1$ for any $x\in
\hat{A}\backslash\{0\}$. Then $|a_n|_T\le 1$ for $n\ge 0$.
\end{lemma}
\begin{proof}
Since $f$ is continuous, it is determined by the values of $f$ at
the points in $A\backslash\{0\}$. From the explanation in the
paragraph after Theorem \ref{thm:Am}, we see that $f$ can be written
as
\begin{equation}\label{e:3.5}
f(x)=\sum\limits_{n=0}^\infty b_n Q_n(x).
\end{equation}
Induction formula (\ref{e:3.3}) and the condition that $|f(x)|_T\le 1$
for any $x\in \hat{A}\backslash\{0\}$ imply $|b_n|_T\le 1$ for all
$n$.

Now we fix a non-negative integer $w$ and let $N\ge r^w-1$ be an integer.
Then for any $\a\neq 0$ with $\deg_T(\a)<w$, we can write the right hand of equation (\ref{e:3.5})
as a finite sum:
\begin{equation*}\label{e:3.6}
f(\a)=\sum\limits_{n=0}^N b_n Q_n(\a).
\end{equation*}
In the above equation, substitute $Q_n(\a)$ by the equation (\ref{e:3.4}) with $R_n(x)=H_n(x)$ for
any non-negative integer $n$, we get
\[
f(\a)=\sum\limits_{j=0}^N\left(\sum\limits_{n=j}^N b_n\gamma_{n,j}\right) H_j(\a)
=\sum\limits_{j=0}^N a_j H_j(\a),
\]
for any $\a\neq 0$ with $\deg_T(\a)<w$. Therefore for any non-negative integer
$j<r^w-1$, Lemma \ref{lem:3.3} implies that
\[
a_j=\sum\limits_{n=j}^N b_n \gamma_{n,j}.
\]
Hence we have $|a_j|_T\le 1$ for $j<r^w-1$. Since $w$ is arbitrary, we get the conclusion.
\end{proof}

\begin{theorem}\label{thm:3.2}
A continuous function $f(x)=\sum\limits_{n=0}^{\infty} a_n G_n(x)$
from ${\F}_r[[T]]$ to ${\F}_r((T))$ is $1$-Lipschitz if and only if
$|a_n|_T\le |T|_T^{\lfloor \log_r (n)\rfloor}$ for $n\ge 1$ and $|a_0|_T\le 1$.
\end{theorem}
\begin{proof}
The proof is very similar to that on the $C^n$ functions
over positive characteristic local rings \cite{Ya}. We can calculate
for $y_1\neq 0$ by using the equation of Proposition
\ref{prop:3.1},
\begin{align}
\begin{split}
\frac 1{y_1}(f(y_1+x)-f(x))
              &=\sum_{n_0=0}^\infty a_{n_0}\frac 1{y_1}
                        (G_{n_0}(y_1+x)-G_{n_0}(x)) \\
              &=\sum_{n_0=0}^\infty\sum_{j_1=0}^\infty\dbinom{n_0+j_1+1}
                {j_1+1}\frac{a_{n_0+j_1+1}}{L_{\nu(j_1+1)}} H_{j_1}(y_1)
                G_{n_0}(x),
\end{split}\label{e:3.6}
\end{align}
The order of summations can be exchanged since the sequence of the
terms in the summation tends to $0$ as $j_1+n_0 \to \infty$ for any
$y_1\neq 0$, and any $x$ in ${\F}_r[[T]]$.

``Sufficiency". The absolute values of $G_{n_0}(x)$, $H_{j_1}(y_1)$,
and the binomial numbers of equation (\ref{e:3.6}) are all less than
or equal to $1$. By Lemma \ref{lem:3.4} and the condition on
$a_n$, we can estimate that $|a_{n_0+j_1+1}/L_{\nu(j_1+1)}|_T\le 1$.
Therefore the function $f$ is $1$-Lipschitz.

``Necessity". Suppose $f$ is $1$-Lipschitz. The function
$\Psi_1f(x,y_1) =\frac 1{y_1}(f(y_1+x)-f(x))$ is continuous on
${\F}_r[[T]]\times ({\F}_r[[T]]\backslash\{0\})$. Since $\Psi_1f(x,y_1)$ is
continuous with respect to
$x\in {\F}_r[[T]]$, we get a function
\[
F_{n_0}(y_1)=
\sum_{j_1=0}^\infty\dbinom{n_0+j_1+1}{j_1+1}\frac{a_{n_0+j_1+1}}{L_{\nu(j_1+1)}}
H_{j_1}(y_1)
\]
for every $n_0\ge 0$. We have $|F_{n_0}(y_1)|_T\le 1$ for any
$y_1\in {\F}_r[[T]]\backslash\{0\}$, since $f$ is $1$-Lipschitz. And $F_{n_0}(y_1)$
is continuous on $\hat{A}\backslash\{0\}$. Then Lemma \ref{lem:3.5}
implies that
\begin{equation}\label{e:3.7}
\left|
\dbinom{n_0+j_1+1}{j_1+1}\frac{a_{n_0+j_1+1}}{L_{\nu(j_1+1)}}\right|_T\le 1,
\end{equation}
for any $n_0\ge 0, j_1\ge 0$.

It is clear that $|a_0|_T\le 1$. For any integer $n\ge 1$, we write
$n=m_0+m_1 r+\cdots + m_w r^w$ in $r$-digit expansion, where $m_w\neq 0$.
We choose non-negative integers $n_0, j_1$ by
\[
j_1+1=l(n)=m_w r^w, \quad \text{ and } \quad n_0 = n-j_1-1.
\]
Then from Lucas formula (\ref{e:lucaslemma}) and Lemma \ref{lem:3.4}, we get
\[
\dbinom{n_0+j_1+1}{j_1+1} =1, \quad\text{ and }\quad |L_{\nu(j_1+1)}|_T=r^{-w}=|T|_T^{\lfloor \log_r (n)\rfloor}.
\]
And from equation (\ref{e:3.7}), we see that
\[
|a_n|_T\le |T|_T^{\lfloor \log_r (n)\rfloor} \quad\text{ for } n\ge 1.
\]
\end{proof}

%====================================================================================================================
\section{Ergodic Functions over ${\F}_2[[T]]$ and Carlitz
Expansions}\label{ecarlitz}

In this section, we take $r=2$. And the Carlitz polynomials $G_n(x)$
and $G'_n(x)$ are defined for $x\in {\F}_2[[T]]$ with coefficients
in ${\F}_2((T))$. We will prove the ergodicity property of functions
over ${\F}_2[[T]]$ by translating the conditions of ergodicity under
Van der Put basis to Carlitz basis. At first we notice that the
polynomials $G_{n}(x)$ and $G_{n}^{\prime}(x)$ have the following
special values:
\begin{itemize}
\item $G_{0}(x)=1$ for any $x$, $G_{n}(0)=0$, if $n\geq 1$;
\item $G_{1}(x)=x$, $G_{n}(1)=0$, if $n\geq 2 $;
\item $G_2(T)=G_2(1+T)=1$, $G_3(T)=T$, $G_3(1+T)=1+T$, and $G_n(T)$ $=$ $G_n(1+T)$ $=0$ if $n\ge
4$;
\item $G_{0}^{\prime}(\alpha)=1$, for any
$\alpha\in\mathbb{F}_{r}[[T]]$.
\end{itemize}
We also recall that $A={\F}_2[T]$, $\hat{A}={\F}_2[[T]]$, $A_n=\{\a\in A:\,\, \deg_T(\a)=n\}$, and
$A_{\le n}=\{\a\in A:\,\, \deg_T(\a)\le n\}$
for any non-negative integer $n$. Moreover, we notice that a function $f\in C(\hat{A},{\F}_2((T)))$
is measure-preserving if and only if
\begin{equation}\label{e:4.1}
\left|\frac{1}{y}\left(f(x+y)-f(x)\right)\right|_T=1
\text{ for any $y\in \hat{A}\backslash\{0\}$ and any $x\in\hat{A}$}.
\end{equation}

\begin{theorem}\label{theorem:car}
\bigskip(ergodicity property) A $1$-Lipschitz function $f:\mathbb{F}_{2}[[T]]\rightarrow
\mathbb{F}_{2}[[T]]$
\[
f(x)=\sum_{n=0}^{\infty}a_{n}G_{n}(x)
\]
is ergodic if and only if the following conditions are satisfied
\begin{itemize}
\item[(1)] $a_{0}\equiv1(\operatorname{mod}T)$, $a_{1}\equiv1+T(\operatorname{mod}
T^{2}),$ $a_{3}\equiv T^{2}(\operatorname{mod}T^{3})$;
\item[(2)] $\left\vert a_{n}\right\vert_T <|T|_T^{\lfloor \log_2(n)\rfloor}=2^{-\lfloor \log_2 (n)\rfloor}$, for $n\geq
2$;
\item[(3)] $a_{2^{n}-1}\equiv T^{n}(\operatorname{mod}T^{n+1})$ for $n\ge 2$.
\end{itemize}
\end{theorem}

\begin{proof}
We have
$f(x)=\sum\limits_{n=0}^{\infty}a_{n}G_{n}(x)=\sum\limits_{\alpha
\in\mathbb{F}_{2}[T]}B_{\alpha}\chi(\alpha,x)$. At first we translate
the conditions (1) and (3) of Theorem \ref{theorem:put} to those on the coefficients of the
Carlitz basis. In the expansion (\ref{e:put}) of Theorem \ref{theorem:put}, we also
use the notation $B_{\a}=T^{\deg_T(\a)} b_{\a}$ for $\a\in {\F}_2[T]\backslash\{0\}$.
\begin{itemize}
\item[(1)] $B_{0}=b_0\equiv1(\operatorname{mod}T)$:

as $B_{0}=f(0)=\sum_{n=0}^{\infty}a_{n}G_{n}(0)$, this condition is equivalent to
\[
a_0=\sum_{n=0}^{\infty}a_{n}G_{n}(0)=f(0)=B_0 \equiv1(\operatorname{mod}T).
\]

\item[] $B_{0}+B_{1}=b_0 + b_1\equiv1+T(\operatorname{mod}T^{2})$:

from $B_{0}+B_{1}=f(0)+f(1)=\sum_{n=0}^{\infty}a_{n}G_{n}(0)+\sum_{n=0}^{\infty
}a_{n}G_{n}(1)$, this condition is equivalent to
\begin{align*}
a_1&\equiv a_0+a_0+a_1\equiv f(0)+f(1)\\
&\equiv B_0+B_1\equiv 1+T (\operatorname{mod} T^2).
\end{align*}

\item[] $b_{T}+b_{1+T}\equiv T(\operatorname{mod}T^{2})$:

this is the same as $B_{T}+B_{1+T}\equiv T^2(\operatorname{mod}T^{3})$. From the explicit calculation
\begin{align*}
B_{T}+B_{1+T}  &  =(f(T)-f(0))+(f(1+T)-f(1))\\
&  =\sum_{n=0}^{\infty}a_{n}(G_{n}(T)-G_{n}(0))+\sum_{n=0}^{\infty}a_{n}%
(G_{n}(1+T)-G_{n}(1))\\
&  =a_{3},
\end{align*}
we see that this condition is equivalent to $a_3\equiv T^2 (\operatorname{mod} T^3)$.

\item[(3)] The third condition of Theorem \ref{theorem:put} (ergodicity property under Van der Put basis)
is $\sum_{\a\in A_{n-1}}b_{\a}\equiv
T(\operatorname{mod}T^{2})$, which is equivalent to
$\sum_{\a\in A_{n-1}}B_{\a}\equiv
T^{n}(\operatorname{mod}T^{n+1})$. We can calculate
\[
\begin{array}{ll}
&\sum\limits_{\a\in A_{n-1}}B_{\a}=\sum\limits_{\b\in A_{\le n-2}}
(f(\b + T^{n-1})-f(\b))\\
=&\sum\limits_{\b\in A_{\le n-2}}\sum\limits_{m=0}^\infty\left(a_m G_m(\b+T^{n-1})-a_m G_m(\b)\right)\\
=&\sum\limits_{\b\in A_{\le n-2}}\sum\limits_{m=1}^\infty a_m \sum\limits_{j=0}^{m-1}
\binom{m}{j} G_j(\b) G_{m-j}(T^{n-1})\\
=&\sum\limits_{m=1}^\infty a_m \sum\limits_{j=0}^{m-1} \binom{m}{j} G_{m-j}(T^{n-1})
\sum\limits_{\b\in A_{\le n-2}} G_j(\b) G'_0(\b)\\
=&\sum\limits_{m=1}^\infty a_m \sum\limits_{j=0}^{m-1} \binom{m}{j} G_{m-j}(T^{n-1})
(-1)^{n-1}\delta_{j,\,2^{n-1}-1}\\
=&a_{2^n-1}.
\end{array}
\]
The last equality of the above equations holds because
$\binom{m}{2^{n-1}-1}\neq 0$ only when $m=(2^{n-1}-1)+l\cdot
2^{n-1}$ for some integer $l\ge 0$ (due to the Lucas formula
(\ref{e:lucaslemma})) and $m-1\ge 2^{n-1}-1$, and we have the
special values of Carlitz polynomials:
\[
G_{l\,\cdot\, 2^{n-1}}(T^{n-1})=
\begin{cases}
1, &\quad\text{if } l=1,\\
0, &\quad\text{if } l>1.
\end{cases}
\]
The order of summation can be exchanged because the function $f$ is
assumed to be $1$-Lipschitz, thus $a_n\to 0$ as $n\to \infty$. Therefore
the third condition of Theorem \ref{theorem:put} is equivalent to
$a_{2^n-1}\equiv T^n \,(\operatorname{mod} T^{n+1})$ for $n\ge 2$.
\end{itemize}

Now we scrutinize equation (\ref{e:3.6}):
\begin{align}
\begin{split}
&\frac 1{y_1}(f(y_1+x)-f(x))\\
              =&\sum_{n_0=0}^\infty\sum_{j_1=0}^\infty\dbinom{n_0+j_1+1}
                {j_1+1}\frac{a_{n_0+j_1+1}}{L_{\nu(j_1+1)}} H_{j_1}(y_1)
                G_{n_0}(x) \\
              = &\, a_1+\sum_{j_1=1}^\infty \frac{a_{j_1+1}}{L_{\nu(j_1+1)}} H_{j_1}(y_1)
                G_{0}(x)\\
               & + \sum_{n_0=1}^\infty\sum_{j_1=0}^\infty\dbinom{n_0+j_1+1}
                {j_1+1}\frac{a_{n_0+j_1+1}}{L_{\nu(j_1+1)}} H_{j_1}(y_1)
                G_{n_0}(x) \label{e:4.2}\\
\end{split}
\end{align}
for $x\in \hat{A}$ and $y_1\in \hat{A}\backslash\{0\}$.

``Sufficiency". Assume the three conditions of the theorem are
satisfied. Then we know that $a_1\equiv 1 (\operatorname{mod} T)$ and we can deduce from
equation (\ref{e:4.2}) that
\begin{equation}\label{e:4.3}
\frac{1}{y_1}\left(f(y_1+x)-f(x)\right)=1+Th(y_1,x),
\end{equation}
where $h(y_1,x)$ is a continuous function from $\left(\hat{A}\backslash\{0\}\right)
\times \hat{A}$ to $\hat{A}$. Therefore $| \frac 1{y_1}$ $(f(y_1+x)-f(x))|_T = 1$
for any $y_1\in \hat{A}\slash\{0\}$, and any $x\in\hat{A}$. Hence the function $f$ is
measure preserving and Theorem \ref{theorem:mp} implies that the condition (2) of Theorem
\ref{theorem:put}, as well as conditions (1) and (3), is satisfied. Therefore the function
$f$ is ergodic.

``Necessity". Assume that the $1$-Lipschitz function $f$ is ergodic. Then by Theorem \ref{theorem:put}
and the discussion at the beginning of this proof, we see that conditions (1) and (3) are satisfied.
We do the same calculation to get equation (\ref{e:4.3}), where
\[
\begin{split}
h(y_1,x)= &\, \frac{1}{T}\Big(a_1-1+\sum_{j_1=1}^\infty \frac{a_{j_1+1}}{L_{\nu(j_1+1)}} H_{j_1}(y_1)
G_{0}(x)\\
 &+ \sum_{n_0=1}^\infty\sum_{j_1=0}^\infty\dbinom{n_0+j_1+1}
{j_1+1}\frac{a_{n_0+j_1+1}}{L_{\nu(j_1+1)}} H_{j_1}(y_1)
G_{n_0}(x)\Big).
\end{split}
\]
As the function $f$ is assumed to be ergodic, $h(y_1,x)$ is a continuous function
from $\left(\hat{A}\backslash\{0\}\right)
\times \hat{A}$ to $\hat{A}$. Therefore we can apply the same proof as Theorem \ref{thm:3.2}
to get the condition (2): $\left\vert a_{n}\right\vert_T <|T|_T^{\lfloor \log_2(n)\rfloor}=
2^{-\lfloor \log_2 (n)\rfloor}$, for $n\geq
2$.
\end{proof}

\begin{example}
In terms of Carlitz basis, the simplest ergodic function on ${\F}_2[[T]]$ would be
\[
f(x)=1 + (1+T)x + \sum\limits_{n=2}^\infty T^n G_{2^n-1}(x).
\]
By Theorem \ref{theorem:car}, we see that an ergodic $1$-Lipschitz function
$f:\mathbb{F}_{2}[[T]]\rightarrow
\mathbb{F}_{2}[[T]]$ can not be a polynomial function.
\end{example}

%=========================================================================================================
\section{Discussion}\label{discussion}

Let $U=\{0,1\}$ and let $k$ be a positive integer. Then $U^k$ is the
set of $k$-tuples consisting of $0$ and $1$. In applications, we
need to study the transformations of $U^k$ for large integers $k$.
For convenience, we can make any element in $U^k$ into an infinite
sequence of $0$ and $1$ by adding all $0$'s after the $k$-th
component of a $k$-tuple:
\[
\begin{array}{rcl}
U^k &\hookrightarrow &U^{\infty}\\
(a_0, a_1, \cdots, a_{k-1}) &\mapsto &(a_0, a_1, \cdots, a_{k-1}, 0,
0, \cdots)
\end{array}
\]
where $U^{\infty}$ is the set of infinite sequences of $0$ and $1$.
So we can study transitive transformations of $U^k$ by analyzing transformations of $U^{\infty}$
and interpreting the results on $U^k$.

Now there are two ways of interpreting the elements of $U^{\infty}$.
One of them is to view an element of $U^{\infty}$ as a $2$-adic
integer, which leads to the work of V. Anashin et al
\cite{An1}\cite{An4} on the theory of measure-preservation and
ergodicity over ${\Z}_2$. The other is to view an element of
$U^{\infty}$ as a power series over the field ${\F}_2$, which leads
to the study of the theory of measure-preservation and ergodicity
over ${\F}_2[[T]]$ in this paper. Therefore we have two different ways to
construct transformations of $U^{\infty}$ which can give rise to
transitive maps from $U^k$ to itself for every positive integer $k$.
Thus it is natural to make comparisons between them, especially when
algorithms are implemented in computer programming, which we will do in the future.

In Theorem \ref{thm:An1}, Theorem \ref{thm:An4}, and Theorem
\ref{thm:An4.2} about the theory over ${\Z}_p$, or Theorem
\ref{theorem:put} and Theorem \ref{theorem:car} about the theory
over ${\F}_2[[T]]$, the formulae of ergodic functions are explicitly
given in respective cases, hence we can write algorithms to
calculate the values of these ergodic functions at any given
elements. While it is certainly important to analyze the efficiency
of these algorithms, it is also worth noticing that the method used
in the proof of Lemma \ref{lem:2.1} could also provide an algorithm
of constructing single cycles of ${\F}_2[[T]]$ modulo $T^{n+1}$ for
any given integer $n\ge 0$. Let $S$ be the set of data
$\{a(k,j)\}_{k\ge 1,\, 0\le j\le 2^k-1}$, where $a(k,j)\in \{0, 1\}$
for each $k$ and each $j$. Given a datum $\{a(k,j)\}_{k,j}$ in $S$
and an integer $n\ge 1$ (in fact the integer $n$ should be large),
we can produce a sequence $\{x_j\}_{0\le j\le 2^{n+1}-1}$ by the
following pseudo codes:
\begin{itemize}
\item[]$x_0=0, x_1=1;$
\item[]{\bf for} $k=1$ {\bf to} $n$ {\bf do}
\item[]$\quad$ {\bf for} $j=0$ {\bf to} $2^{k}-1$ {\bf do}
\item[]$\quad\quad$ {\bf if} $a(k,j)=1$ {\bf then} $x_j=x_j+T^{k}$;
\item[]$\quad$ {\bf end for};
\item[]$\quad$ {\bf for} $j=2^{k}$ {\bf to} $2^{k+1}-1$ {\bf do}
\item[]$\quad\quad$ $x_j=x_{j-2^{k}} + T^{k}$;
\item[]$\quad$ {\bf end for};
\item[]{\bf end for};
\end{itemize}
By the process of the proof of Lemma \ref{lem:2.1}, the sequence $\{x_j\}_{0\le j\le 2^{n+1}-1}$ gives a
transitive transformation $f$ of ${\F}_2[[T]]/(T^{n+1}\cdot {\F}_2[[T]])$ such that
$f \mod T^k$ is transitive for every $k=1, 2, \cdots, n+1$,
by defining $f(x_j)=x_{j+1}$ for $0\le j\le 2^{n+1}-2$ and $f(x_{2^{n+1}-1})=x_0$.
From the proof of Lemma \ref{lem:2.1}, we can also see that the above construction gives
all ergodic $1$-Lipschitz functions from ${\F}_2[[T]]$ to itself,
hence the data $\{a(k,j)\}_{1\le k\le n,\, 0\le j\le 2^k-1}$
give rise to all transitive transformations of ${\F}_2[[T]]/(T^{n+1}\cdot {\F}_2[[T]])$ which are
also transitive modulo $T^k$ for every $k=1, 2, \cdots, n+1$.

It might be useful to generalize the results of this paper to the ergodic theory over
${\F}_r[[T]]$ for $r$ being a power of a prime number. To do this, the main difficulty
is that we don't have a description of ergodicity as simple as Lemma \ref{lem:2.1}.
So this problem will have to be left for studies in the future.

\bigskip
\end{document}